\documentclass[12pt]{article}
\usepackage{MnSymbol}
\usepackage{mathrsfs}
\usepackage[utf8]{inputenc}
\usepackage{amsmath,amsthm}

\usepackage{tikz-cd}

\theoremstyle{plain}
\newtheorem{Theorem}{Theorem}
\newtheorem{Lemma}{Lemma}
\newtheorem{Proposition}[Lemma]{Proposition}
\newtheorem{Corollary}[Lemma]{Corollary}
\theoremstyle{definition}

\numberwithin{equation}{section}

\DeclareMathAlphabet\mathbb{U}{msb}{m}{n}

\begin{document}

\begin{center}
{\Large \bf Higher Jacobi identities  \\} 
\ \\ I. Alekseev and S.O. Ivanov  
\end{center}

\

\begin{center}
\begin{minipage}{300pt}  {\sc Abstract}. \footnotesize
By definition the identities $[x_1,x_2]+[x_2,x_1]=0$ and $[x_1,x_2,x_3]+[x_2,x_3,x_1]+[x_3,x_1,x_2]=0$ hold in any Lie algebra. It is easy to check that the identity $[x_1,x_2,x_3,x_4]+[x_2,x_1,x_4,x_3]+[x_3,x_4,x_1,x_2]+[x_4,x_3,x_2,x_1] = 0$ holds in any Lie algebra as well. We investigate sets of permutations that give identities of this kind. In particular, we construct a family of such subsets $T_{k,l,n}$ of the symmetric group $S_n,$ and hence, a family of identities that hold in any Lie algebra. 
\end{minipage}
\end{center}

\

\section*{Introduction}

By a Lie ring we mean a Lie algebra over $\mathbb Z.$ Any Lie algebra can be considered as a Lie ring. By definition the identities $[x_1,x_2]+[x_2,x_1]=0$ and $[x_1,x_2,x_3]+[x_2,x_3,x_1]+[x_3,x_1,x_2]=0$ hold in any Lie ring, where $[x_1,\dots,x_n]$ denotes the left-normed bracket. 
Moreover, it is easy to check that there is one more identity that holds in any Lie ring: $[x_1,x_2,x_3,x_4]+[x_2,x_1,x_4,x_3]+[x_4,x_3,x_2,x_1]+[x_3,x_4,x_1,x_2]=0.$ This motivates the following definition. 
A subset $T$ of the symmetric group $S_n$ is said to be {\it Jacobi} if the following identity holds in any Lie ring 
\begin{equation}\label{id_T}
\sum_{\sigma\in T} [x_{\sigma(1)},\dots,x_{\sigma(n)}]=0.
\end{equation} 
The paper is devoted to investigation of Jacobi  subsets. 
A family of Jacobi subsets $T_{k,l,n}\subseteq  S_n,$ where $k+l\leq n,$ was constructed (Theorem \ref{Theorem1}). Moreover, a `vector space of all relations between left-normed brackets' was constructed and a basis of the space induced by the subsets $T_{k,1,n}$ was found (Theorem \ref{Theorem2}). The most interesting identities come from the Jacobi subsets $T_{k,l}:=T_{k,l,k+l}.$ Here we give some examples of identities that hold in any Lie ring and come from the sets $T_{k,l}:$
\begin{itemize}
\item $T_{1,1}: \hspace{0.5cm} [x_1,x_2]+[x_2,x_1] = 0;$
\item $T_{1,2}: \hspace{0.5cm} [x_1,x_2,x_3]+[x_2,x_3,x_1]+[x_3,x_1,x_2] = 0;$
\item $T_{2,2}:\hspace{0.5cm} [x_1,x_2,x_3,x_4]+[x_2,x_1,x_4,x_3]+[x_3,x_4,x_1,x_2]+[x_4,x_3,x_2,x_1] = 0;$
\item $T_{1,3}: \hspace{0.5cm}[x_1,x_2,x_3,x_4]+[x_3,x_1,x_2,x_4]+[x_4,x_1,x_2,x_3]+[x_1,x_4,x_3,x_2]+[x_2,x_3,x_4,x_1] = 0;$
\item 
$T_{2,3}: \hspace{0.5cm}$ 
\begin{minipage}{11cm}
$[x_1,x_2,x_3,x_4,x_5]+[x_2,x_1,x_4,x_3,x_5]+[x_2,x_1,x_5,x_3,x_4]+$ \\
$[x_1,x_2,x_5,x_4,x_3]+[x_3,x_4,x_5,x_1,x_2]+[x_4,x_3,x_5,x_2,x_1] = 0;$
\end{minipage}

\item
$T_{3,3}:\hspace{0.5cm}$ 
\begin{minipage}{13cm}$
[x_1,x_2,x_3,x_4,x_5,x_6]+[x_2,x_1,x_3,x_5,x_4,x_6]+[x_2,x_1,x_3,x_6,x_4,x_5]+$\\
$[x_1,x_2,x_3,x_6,x_5,x_4]+
[x_4,x_5,x_6,x_1,x_2,x_3]+[x_5,x_4,x_6,x_2,x_1,x_3]+$ \\ 
$[x_5,x_4,x_6,x_3,x_1,x_2]+[x_4,x_5,x_6,x_3,x_2,x_1] = 0.$
\end{minipage}
\end{itemize} 
There are several simple operations over Jacobi subsets that allow to obtain new Jacobi subsets (Lemma \ref{Lemma_properties_of_Jacobi}). Using this we can obtain a big amount of Jacobi subsets from the sets $T_{k,l,n}$ but not all of them. For example, the following identity 
$$[x_1,x_2,x_3,x_4]+[x_3,x_1,x_2,x_4]+[x_4,x_1,x_2,x_3]+[x_1,x_4,x_3,x_2]+[x_4,x_3,x_2,x_1]+[x_2,x_4,x_3,x_1] = 0$$
holds in any Lie ring but the corresponding subset of $S_4$ can not be obtained from the sets $T_{k,l,4}$ using these operations. 

Another object of our interest is the set of $2$-Jacobi subsets. A subset $T\subseteq S_n$ is said to be $2$-Jacobi if the identity \eqref{id_T} holds in any Lie algebra over a field of characteristic $2$. Of course, any Jacobi subset is 2-Jacobi. But there are a lot of 2-Jacobi subsets which are not Jacobi subsets. For example, $\{(),(123), (13)\}.$ Moreover, it is easy to check that for $n=3$ the number Jacobi subsets is $10$ but the number of $2$-Jacobi subsets is $16.$ The advantage of 2-Jacobi subsets is that it is easier to investigate them. For example, we do not know the number of Jacobi subsets of $S_n$ but we know the number of 2-Jacobi subsets: $2^{(n-1)!\cdot (n-1)}$ (Corollary \ref{Corollary_number_2-Jacobi}).   
It was proved that, in contrast to the class of usual Jacobi subsets, the class of $2$-Jacobi 
subsets is closed under symmetric difference, which makes it a $\mathbb Z/2$-vector space. Moreover, it was proved that any $2$-Jacobi subset can be obtained as a symmetric difference of several Jacobi subsets.

The paper is organised as follows: in Section \ref{section_Results} we present all definitions, constructions and results without proofs; in Section \ref{section_Proofs} we give all proofs.

\section{Results}\label{section_Results}

The left normed Lie bracket of elements $a_1,\dots a_n$ of a Lie ring $L$ is defined by recursion  $[a_1,\dots a_n]:=[[a_1,\dots, a_{n-1}],a_n],$ where $[a_1]=a_1.$ 
By $S_n$ we denote the symmetric group on $\{1,\dots,n\}.$ If $n\leq m$ we denote by 
$$\iota_{n,m}:S_n\hookrightarrow S_m$$
the canonical embedding. 

\subsection{Jacobi subsets} 

 A subset $T$ of the symmetric group $S_n$ is said to be {\it Jacobi} if the following identity holds in any Lie ring 
\begin{equation}
\sum_{\sigma\in T} [a_{\sigma(1)},\dots,a_{\sigma(n)}]=0.
\end{equation} 
It is easy to see that $\{(),(1,2)\}$ and $\{ (),(1,2,3), (1,3,2) \}$ are Jacobi subsets of $S_n$ for any $n\geq 3.$ The following Lemma gives more examples of Jacobi subsets.

\begin{Lemma}\label{Lemma_properties_of_Jacobi} Let $T,T'$ be Jacobi subsets of $S_n$ and $H\subseteq G$ be subgroups of $S_n.$ Then the following holds. 
\begin{enumerate}
\item If  $T\cap T'=\emptyset,$ then $T\cup T'$ is Jacobi. 

\item  $\sigma T$ is Jacobi for any $\sigma\in S_n$.

\item $T (1,2)$ is Jacobi, where $(1,2)\in S_n$ is the transposition.

\item If $n\leq m$, then $\iota_{n,m}(T)$ is a Jacobi subset of $S_m$.

\item If $H$ is Jacobi, then $G$ is Jacobi. 
\item If $n\geq 2$ and $(1,2)\in G,$ then $G$ is Jacobi.
\item If $n\geq 3$ and $(1,2,3)\in G,$ then $G$ is Jacobi.  
\end{enumerate}
\end{Lemma}

Let $\mathbb Z\langle x_1,\dots, x_n \rangle$ be the free associative ring generated by elements $x_1,\dots,x_n.$ We denote by $\gamma_n$ the additive subgroup of  $\mathbb{Z}\langle x_1,\dots,x_n\rangle$ generated by the monomials  $x_{\sigma(1)}\dots x_{\sigma(n)}$ where $\sigma$ runs over $S_n.$ It is easy to see that $\gamma_n$ is a free abelian group of rank $n!$ and 
the monomials $x_{\sigma(1)}\dots x_{\sigma(n)}$ form its basis. Define a homomorphism $\beta_n:\gamma_n\to \gamma_n$ on the basis by the formula 
$$\beta_n(x_{\sigma(1)}\dots x_{\sigma(n)})=[x_{\sigma(1)},\dots ,x_{\sigma(n)}].$$
For a subset $T\subseteq S_n$ we set $${\sf Sum}(T)=\sum_{\sigma\in T} x_{\sigma(1)}\dots x_{\sigma(n)}\ \in \gamma_n.$$
\begin{Lemma}\label{Lemma_equivalent_Jacobi}
A subset $T\subseteq S_n$ is Jacobi if and only if ${\sf Sum}(T) \in {\rm Ker}(\beta_n).$
\end{Lemma}

\subsection{Shuffles and sets $T_{k,l,n}$}

An {\it $(s,t)$--shuffle} is a pair $(\alpha,\beta)$ such that $\alpha:\{1,\dots,s\} \to \{1,\dots, s+t\}$ and $\beta:\{1,\dots,t\} \to \{1,\dots, s+t\}$ are strictly monotonic functions with disjoint images. The set of all $(s,t)$--shuffles is denoted by ${\sf Sh}(s,t)$. The set of all $(s,t)$--shuffles such that $\alpha(1)=1$ is denoted by ${\sf Sh}^1(s,t).$

\begin{Proposition}\label{Proposition_bracket_shaffles} Let $L$ be a Lie ring and $a,a_1,\dots a_n \in L.$ Then
$$[a,[a_{1}, \ldots, a_{n}]] = \sum\limits_{i=0}^{n-1}\sum\limits_{(\alpha, \beta)}(-1)^i \left[a, a_{\beta(i)},\dots, a_{\beta(1)},a_{\alpha(1)}, \dots, a_{\alpha(n-i)}\right],$$ 
where the second sum is taken over all shuffles $(\alpha,\beta) \in {\sf Sh}^1(n-i,i)$.
\end{Proposition}

Let $k,l\geq 1$ be natural numbers, $0\leq i\leq l-1$  and $(\alpha, \beta)\in {\sf Sh}^1(l-i,i)$. Consider the following permutations of $S_{k+l}$
$$\tilde{\sigma}_{\alpha,\beta,k,l} = \left(\begin{array}{cccccccccccc} 
1 & \cdots & k & k+1        & \cdots & k+i   & k+i+1       & \cdots & k+l    \\
1 & \cdots & k & k+\beta(i) & \cdots & k+\beta(1) & k+\alpha(1) & \cdots & k+\alpha(l-i)
\end{array}\right).$$
and $\sigma_{\alpha,\beta, k,l} = \tilde{\sigma}_{\alpha,\beta,k,l}\circ (1,2)^i.$
The set of all such permutations $\sigma_{\alpha,\beta,k,l}$ with fixed $k$ and $l$ is denoted by 
$$C_{k,l}=\{ \sigma_{\alpha,\beta,k,l}\mid 0\leq i\leq l-1, (\alpha,\beta)\in {\sf Sh}^1(l-i,i) \}.$$
\begin{Lemma}\label{lemma_bracket_sum} For any  $k,l\geq 1$ the equation $$[[x_1,\dots,x_k],[x_{k+1},\dots,x_{k+l}]]=\beta_{k+l}({\sf Sum}(C_{k,l}))$$
holds in $\mathbb Z\langle x_1,\dots,x_{k+l} \rangle$. 
\end{Lemma}
For natural numbers $k$ and $l$ we consider a permutation $\Phi_{k,l}\in S_{k+l}$ given by 
$$\Phi_{k,l}(i)=\begin{cases} 
i+k, & \text{ if } i\leq l,\\
i-l, & \text{ if } i>l.
\end{cases}$$
Roughly speaking, $\Phi_{k,l}$ shifts the interval $\{1,\dots,l\}$ on the place of the interval $\{k+1,\dots, k+l\}$ and shifts the interval $\{l+1,\dots,k+l\}$ on the place of the interval $\{1,\dots, k\}. $  
\begin{Corollary}\label{Corollary_bracket_sum} For any $k,l\geq 1$ the equation $$[[x_{k+1},\dots,x_{k+l}],[x_{1},\dots,x_{k}]]=\beta_{k+l}({\sf Sum}(\Phi_{k,l}C_{l,k}))$$
holds in $\mathbb Z\langle x_1,\dots,x_{k+l} \rangle$. 
\end{Corollary}
For $k,l\geq 1$ and $k+l\leq n$ we set
$$T_{k,l}:=C_{k,l} \cup (\Phi_{k,l} C_{l,k}), \hspace{1cm} T_{k,l,n}:=\iota_{k+l,n}(T_{k,l}).$$
\begin{Lemma}\label{Lemma_do_not_intersect} For $k,l\geq 1$ the subsets $C_{k,l}$ and $\Phi_{k,l}C_{l,k}$ do not intersect.
\end{Lemma}

\begin{Theorem}\label{Theorem1} For $k,l\geq 1$ and $k+l\leq n$ the set $T_{k,l,n}$ is a Jacobi subset of $S_n.$
\end{Theorem}

\begin{Theorem}[{cf. 8.6.7 of \cite{Reutenauer}}]\label{Theorem2} For any $n\geq 2$ the set $$\{{\sf Sum}(\sigma T_{k,1,n}) \mid \sigma\in S_n, \  \sigma(k+1)=1,\  1\leq k \leq n-1 \}$$
is a basis of the free abelian group ${\sf Ker}(\beta_n).$ In particular, the rank of ${\sf Ker}(\beta_n)$ equals to \hbox{$(n-1)!\cdot (n-1).$}
\end{Theorem}

\subsection{2-Jacobi subsets} 

A subset $T$ of the symmetric group $S_n$ is said to be {\it 2-Jacobi} if the identity \eqref{id_T} holds in any Lie algebra over the field $\mathbb Z/2$ (equivalently, over a field of characteristic 2). Of course, any Jacobi subset is 2-Jacobi. But there are a lot of 2-Jacobi subsets which are not Jacobi subsets. For example, $\{(),(123), (13)\}.$ Moreover, it is easy to check that for $n=3$ the number Jacobi subsets is $10$ but the number of $2$-Jacobi subsets is $16.$ The advantage of 2-Jacobi subsets is that it is easier to investigate them. For example, we do not know the number of Jacobi subsets of $S_n$ but we know the number of 2-Jacobi subsets: $2^{(n-1)!\cdot (n-1)}.$   

\begin{Proposition}\label{Proposition_2-Jacobi} The set of 2-Jacobi subsets of $S_n$ is closed under symmetric difference. Moreover, the set of 2-Jacobi subsets of $S_n$ is a vector space over $\mathbb Z/2$ with respect to the symmetric difference with a  basis is given by $$\{\sigma T_{k,1,n} \mid \sigma\in S_n, \  \sigma(k+1)=1,\  1\leq k \leq n-1 \}.$$ 
\end{Proposition}

\begin{Corollary}\label{Corollary_number_2-Jacobi} The number of 2-Jacobi subsets of $S_n$ equals to $2^{(n-1)!\cdot (n-1)}.$
\end{Corollary}
\begin{Corollary}
Any 2-Jacobi subset can be presented as a symmetric difference of several Jacobi subsets. 
\end{Corollary}

\section{Proofs}\label{section_Proofs}

\begin{proof}[Proof of Lemma \ref{Lemma_properties_of_Jacobi}] (1) Obvious. 

(2) Since $T$ is Jacobi, we have $\sum_{\tau \in T} [a_{\tau(1)},\dots,a_{\tau(n)}]=0$ for any elements $a_1,\dots, a_n$ of any Lie ring $L.$ Set $a_i=b_{\sigma(i)}.$ Then $0=\sum_{\tau \in T} [a_{\tau(1)},\dots,a_{\tau(n)}]=\sum_{\tau \in T} [b_{\sigma\tau(1)},\dots,b_{\sigma\tau(n)}]=\sum_{\tau \in \sigma T} [b_{\tau(1)},\dots,b_{\tau(n)}]$ for any elements $b_1,\dots,b_n$ of any lie ring $L.$ Then $\sigma T$ is Jacobi.

(3) $\sum_{\tau \in T(1,2)} $ 
$ [a_{\tau(1)},a_{\tau(2)}, \dots,$ 
$a_{\tau(n)}]$ 
$=$ 
$\sum_{\tau \in T} $
$[a_{\tau(2)}, a_{\tau(1)}\dots,$ 
$a_{\tau(n)}]$ 
$=$ 
$-(\sum_{\tau \in T} $
$[a_{\tau(1)}, a_{\tau(2)}\dots,$
$a_{\tau(n)}])$ 
$=$ 
$0.$
 
(4) Consider any elements $a_1,\dots,a_n, \dots , a_m.$ Then  $\sum_{\tau \in T} [a_{\tau(1)},\dots,a_{\tau(n)}]=0,$ 
and hence $0$ 
$=$ 
$[[\sum_{\tau \in T}$ 
$[a_{\tau(1)},\dots,a_{\tau(n)}]$ 
$,$ 
$a_{n+1}, \dots , a_{m} ] $ 
$=$ 
$\sum_{\tau \in T}$ 
$[a_{\tau(1)},\dots,a_{\tau(n)},$ 
$a_{n+1},\dots,a_m]$ 
$=$ 
$\sum_{\tau \in \iota_{n,m}(T)}$ 
$[a_{\tau(1)},\dots,a_{\tau(n)}]. $

(5) By (2) we get that all cosets $gH$ are Jacobi for any $g\in G$, and by (1) we obtain that their union is Jacobi. 

(6) and (7) follow from (5).
\end{proof}
\begin{proof}[Proof of Lemma \ref{Lemma_equivalent_Jacobi}] 
Denote by $L(x_1,\dots,x_n)$ the Lie subring of $\mathbb Z\langle x_1,\dots,x_n \rangle$ generated by $x_1,\dots,x_n.$ The Lie ring $L(x_1,\dots,x_n)$ is the free Lie ring on $x_1,\dots,x_n$ \cite[Theorem 0.5]{Reutenauer}. Then for any elements $a_1,\dots,a_n$ of a Lie ring $L$ there exist a unique Lie algebra homomorphism $f:L(x_1,\dots,x_n)\to L$ such that $f(x_i)=a_i.$ It follows that $T\subseteq S_n$ is Jacobi if and only if $\sum_{\sigma\in T}[x_{\sigma(1)},\dots, x_{\sigma(n)}]=0$ in $\mathbb Z\langle x_1,\dots x_n \rangle.$ Then the statement follows from the equality 
 $\beta_n({\sf Sum}(T))=\beta_n(\sum_{\sigma \in T} a_{\sigma(1)}\dots a_{\sigma(n)})=\sum_{\tau \in T} [a_{\tau(1)},\dots,a_{\tau(n)}]$.
\end{proof}
\begin{proof}[Proof of Proposition \ref{Proposition_bracket_shaffles}]
First we prove the following statement that seems to be known but we can not find a good reference. 
Let $R$ be an associative ring and $a_1,\dots,a_n\in R$. Then 
\begin{equation}\label{eq_sh_com_assotiative}
[a_1,\dots, a_n]=\sum_{i=0}^{n-1} \ \sum_{(\alpha,\beta)} (-1)^i a_{\beta(i)}\dots a_{\beta(1)}a_{\alpha(1)}\dots a_{\alpha(n-i)},
\end{equation}
where the second sum is taken  $(\alpha,\beta)\in {\sf Sh}^1(n-i,i).$ The prove is by induction. For $n=2$ it is obvious. Assume that the formula holds for $[a_1,\dots,a_n]$ and prove it for $[a_1,\dots,a_{n+1}].$ The element $[a_1,\dots,a_{n+1}]$ is the sum of elements 
$$(-1)^i a_{\beta(i)}\dots a_{\beta(1)}a_{\alpha(1)}\dots a_{\alpha(n-i)}a_{n+1}+(-1)^{i+1} a_{n+1}a_{\beta(i)}\dots a_{\beta(1)}a_{\alpha(1)}\dots a_{\alpha(n-i)},$$
where sum is taken over $0\leq i<n$ and $(\alpha,\beta)\in {\sf Sh}(n-i,i).$ Any $(n+1-i,i)$-shuffle $(\alpha(1),\dots,\alpha(n+1-i);\beta(1),\dots, \beta(i))$ is equal to either  $(\alpha'(1),\dots,\alpha'(n-i),n+1;\beta(1),\dots, \beta(i))$ for a $(n-i,i)$-shuffle $(\alpha,\beta')$ or $(\alpha(1),\dots,\alpha(n+1-i);\beta'(1),\dots, \beta'(i-1),n+1)$ for a $(n+1-i,i-1)$-shuffle $(\alpha,\beta').$ The assertion follows.

Now we use the formula \eqref{eq_sh_com_assotiative} to prove the proposition. Fix a Lie ring $L.$ For an endomorphism $\varphi \in {\sf End}(L)$ and $a\in L$ we set $a.\varphi=\varphi(a).$ For any two endomorphisms $\varphi,\psi \in {\sf End}(L)$ we write $\varphi * \psi =\psi \circ \varphi.$ Then $(a.\varphi) . \psi=a.(\varphi * \psi).$  The commutator of  $\varphi$ and $\psi$ with respect to $*$ is denoted by $[\varphi,\psi]_*=\varphi*\psi - \psi*\varphi.$ We consider ${\sf End}(L)$ as a ring with the operation $*$ which is opposite to the composition. Consider the map ${\sf ad}':L \to {\sf End}(L)$ given by ${\sf ad}'(a)(b)=[b,a].$ Then ${\sf ad}'$ is a homomorphism of Lie algebras i.e. ${\sf ad}'([a,b])=[{\sf ad}'(a),{\sf ad}'(b)]_*.$ Let $a_1,\dots,a_n$ be elements of $L.$ Set ${\bar a}_i={\sf ad}'(a_i).$ The equation \eqref{eq_sh_com_assotiative} implies that 
$$[{\bar a}_1,\dots, {\bar a}_n]_*=\sum_{i=0}^{n-1} \ \sum_{(\alpha,\beta)} (-1)^i {\bar a}_{\beta(i)}*{\dots} *{\bar a}_{\beta(1)}*{\bar a}_{\alpha(1)}*{\dots} *{\bar a}_{\alpha(n-i)},$$
where the second sum is taken  $(\alpha,\beta)\in {\sf Sh}^1(n-i,i).$ If we apply $a.-$ to both parts of the equality, we obtain the required formula. 
\end{proof}
\begin{proof}[Proof of Lemma \ref{lemma_bracket_sum}]
Proposition \ref{Proposition_bracket_shaffles} implies that $[[x_1,\dots,x_k],[x_{k+1},\dots,x_{k+l}]]$ equals to the sum of elements 
\begin{equation}
\begin{split}
&(-1)^i[x_1,\dots,x_k,x_{k+\beta(i)},\dots,x_{k+\beta(1)},x_{k+\alpha(1)},\dots,x_{k+\beta(n-i)}]= \\
& (-1)^i[x_{\tilde \sigma_{\alpha,\beta,k,l}(1)}, \dots ,x_{\tilde \sigma_{\alpha,\beta,k,l}(k+l)}]=
\\
&  [x_{ \sigma_{\alpha,\beta,k,l}(1)}, \dots ,x_{ \sigma_{\alpha,\beta,k,l}(k+l)}].
\end{split}
\end{equation}
The assertion follows. 
\end{proof}
\begin{proof}[Proof of Lemma \ref{Lemma_do_not_intersect}] Consider cases.

Let $k\geq 2$ and $l\geq 1.$ Then for any $\sigma\in C_{k,l}$ we have $\sigma(\{1,2\})=\{1,2\}$ and for any $\tau\in \Phi_{k,l} C_{l,k}$ we have $k+1 \in \tau(\{1,2\}).$ Then $C_{k,l}\cap \Phi_{k,l}C_{l,k}=\emptyset.$ 

Let $k=1$ and $l\geq 2.$ Then $C_{l,k}=C_{l,1}=\{()\},$ and hence $\Phi_{1,l}C_{l,1}=\{\Phi_{1,l}\}$. For any $\sigma\in C_{1,l}$ we have $1\in \sigma(\{1,2\})$ but $\Phi_{1,l}(\{1,2\})=\{2,3\}.$ Then $C_{1,l}\cap \Phi_{1,l}C_{l,1}=\emptyset.$ 

Let $k=1$ and $l=1.$ Then $C_{1,1}=\{()\}$ and $\Phi_{1,1}C_{1,1}=\{(1,2)\}.$
\end{proof}

\begin{proof}[Proof of Theorem \ref{Theorem1}]
Since $C_{k,l}$ and $\Phi_{k,l} C_{l,k}$ are disjoint, we get  $\beta_n({\sf Sum}(T_{k,l}))$ $=$ $\beta_n({\sf Sum}(C_{k,l}))$ $+$ $\beta_n({\sf Sum}(\Phi_{k,l}C_{l,k}))$ $=$ $[[x_1,\dots,x_k] , [x_{k+1},\dots, x_{k+l}]]$ $+$ $[[x_{k+1},\dots , x_{k+l}] , [x_1,\dots,x_k]]$ $=$ $0.$
\end{proof}

\begin{proof}[Proof of Theorem \ref{Theorem2}]
It is proved in \cite{Blessenohl-Laue} (see also \cite[p. 211]{Reutenauer}) that the set $\{\sigma \theta_{j,n} \mid \sigma(j)=1, 2\leq j \leq n  \}$ is a basis of ${\rm Ker}(\beta_n)$, where 
$$\theta_{j,n}=x_1\dots x_n + x_j[x_1,\dots,x_{j-1}]x_{j+1}\dots x_n.$$
Prove that $\theta_{j,n}={\sf Sum}(T_{j-1,1,n}).$ Since $\theta_{j,n}$ is the image of $\theta_{j,j},$ it is sufficient to prove it for $j=n.$  Note that ${\sf Sum}(T_{n-1,1,n})={\sf Sum}(C_{n-1,1}) +{\sf Sum}(\Phi_{n-1,1}C_{1,n-1}) $ and $C_{n-1,1}=\{()\}.$ It follows that 
\begin{align*}
&{\sf Sum}\left(T_{n-1, 1, n}\right) = x_1 x_2 \dots x_{n} + \sum_{i=0}^{n-2}\sum_{(\alpha, \beta)}(-1)^i x_n x_{\beta(i)} \dots x_{\beta(1)} x_{\alpha(1)}  \dots x_{\alpha(n-1-i)} =\\
&\ \hspace{-0.5cm}= x_1 x_2 \dots x_{j} + x_n\sum_{i=0}^{n-2}\sum_{(\alpha, \beta)}(-1)^i x_{\beta(i)} \dots x_{\beta(1)} x_{\alpha(1)}  \dots x_{\alpha(n-1-i)} =\theta_{n,n}.
\end{align*}
where the second sum is taken over all shuffles $(\alpha,\beta) \in {\sf Sh}^1(n-1-i,i)$.
\end{proof}

\begin{proof}[Proof of Proposition \ref{Proposition_2-Jacobi}]
Set 
$$\gamma_n^{2}=\gamma_n\otimes \mathbb Z/2, \hspace{1cm}  \beta_n^{2}=\beta_n\otimes {\sf id}_{\mathbb Z/2}, \hspace{1cm} {\sf Sum}^{2}(T)={\sf Sum}(T)\otimes 1 \in \gamma_n^2$$ for any $T\subseteq S_n.$ Similarly to Lemma \ref{Lemma_equivalent_Jacobi} one can prove that a subset $T\subseteq S_n$ is 2-Jacobi if and only if $\beta^2_n({\sf Sum}^2(T))=0.$ The set of subsets $\mathcal P(S_n)$ of the symmetric group $S_n$ is a vector space over $\mathbb Z/2$ with respect to the symmetric difference. Moreover, it is easy to see that the map ${\sf Sum}^2:\mathcal P(S_n) \to \gamma^2_n$ is an isomorphism of vector spaces over $\mathbb Z/2$.  Hence the set of 2-Jacobi subsets is the kernel of the homomorphism $\beta^2_n\circ {\sf Sum}^2.$ It follows that 2-Jacobi subsets are closed under the symmetric difference. Moreover, ${\sf Sum}^2$ induces an isomorphism between the vector space of 2-Jacobi subsets and ${\rm Ker}(\beta^2_n).$ Theorem \ref{Theorem2} implies that $\{{\sf Sum}^2(\sigma T_{k,1,n}) \mid \sigma\in S_n, \  \sigma(k+1)=1,\  1\leq k \leq n-1 \}$ is a basis of $\gamma^2.$ Thus $\{\sigma T_{k,1,n} \mid \sigma\in S_n, \  \sigma(k+1)=1,\  1\leq k \leq n-1 \}$ is a basis of the vector space of 2-Jacobi subsets. 
\end{proof}

\end{document}